\documentclass[12pt]{amsart}
\usepackage{amssymb,amsmath,amscd,graphicx,
latexsym,amsthm}
\usepackage{amssymb,latexsym,eufrak,amsmath,amscd,graphicx}
\usepackage[mathscr]{eucal}
  \usepackage[all]{xy}
\theoremstyle{plain}
\newtheorem{theorem}{Theorem}[section]

\newtheorem{proposition}[theorem]{Proposition}
\newtheorem{corollary}[theorem]{Corollary}
\newtheorem{lemma}[theorem]{Lemma}
\newtheorem{variant}[theorem]{Variant}

\numberwithin{equation}{section}

\newtheorem{theoremalpha}{Theorem}
\newtheorem{corollaryalpha}[theoremalpha]{Corollary}

\theoremstyle{definition}

\newtheorem{remark}[theorem]{Remark}

\newcommand{\lra}{\longrightarrow}
\newcommand{\noi}{\noindent}

\newcommand{\QQ}{\mathbf{Q}}

\newcommand{\OO}{\mathcal{O}}

\newcommand{\cP}{\mathcal{P}}

\newcommand{\II}{\mathcal{I}}

\newcommand{\eps}{\varepsilon}

\newcommand{\HH}[3]{H^{{#1}} \big( {#2} , {#3}
\big) }
\newcommand{\HHH}[3]{H^{{#1}} \Big( {#2} , {#3}
\Big) }

\newcommand{\MI}[1]{\mathcal{J}  ( {#1}
) }

\newcommand{\mult}{\textnormal{mult}}
\newcommand{\pr}{\prime}

\newcommand{\num}{ \equiv_{\text{num}} }

\newcommand{\Bl}{\text{Bl}}

\newcommand{\Linser}[1]{| \mspace{1.5mu} {#1}
\mspace{1.5mu} |}
\newcommand{\linser}[1]{\Linser{  {#1}  }}

 \newcommand{\pro}{\textnormal{pr}}

\newcommand{\opbox}{\mathop{\boxtimes}}

\begin{document}

\title{Local positivity, multiplier ideals,  and syzygies of abelian varieties}

\author{Robert Lazarsfeld}
  \address{Department of Mathematics, University of Michigan, Ann Arbor, MI
   48109}
 \email{{\tt rlaz@umich.edu}}
 \thanks{Research of the first author partially supported by NSF grant DMS-0652845.}

\author{Giuseppe Pareschi}
  \address{Dipartimento di matematica, Universit\`a di Roma, Tor Vergata, V.le della Ricerca Scientifica, I-00133 Roma, Italy}
 \email{{\tt pareschi@mat.uniroma2.it}}

\author{Mihnea Popa}
\address{Department of Mathematics, University of Illinois at Chicago, 851 S. Morgan Street, Chicago, IL 60607}
\email{\tt{mpopa@math.uic.edu}}
 \thanks{Research of the third author partially supported by NSF grant DMS-0758253 and  a Sloan Fellowship.}

\maketitle

\section*{Introduction}

In a recent paper \cite{HT}, Hwang and To observed that there is a relation between local positivity on an abelian variety $A$  and  the projective normality of suitable embeddings of $A$. The purpose of this note is to extend their result to higher syzygies, and to show that the language of multiplier ideals renders the computations extremely quick and transparent. 

Turning to details, let $A$ be an abelian variety of dimension $g$, and let $L$ be an ample line bundle on $A$. Recall that the \textit{Seshadri constant} $\eps(A, L)$ is a positive real number that measures the local positivity of $L$ at any given point $x \in A$: for example, it can be defined by counting asymptotically the number of jets that the linear series $\linser{kL}$ separates at $x$ as $k \to \infty$. We refer  to \cite[Chapter 5]{PAG} for a general survey of the theory, and in particular to Section 5.3 of that book for a discussion of local positivity on abelian varieties.

Our main result is 
\begin{theoremalpha} \label{Main.Thm}
Assume that \[ \eps(A,L) \ > \ (p+2) g. \]
Then $L$ satisfies property $(N_p)$.
\end{theoremalpha}
\noi The reader may consult for instance  \cite[Chapter 1.8.D]{PAG}, \cite{GL} or    \cite{Eisenbud} for the definition of property $(N_p)$ and further references. Suffice it  to say here that   $(N_0)$ holds when $L$ defines a projectively normal embedding of $A$, while $(N_1)$ means that the homogeneous ideal of    $A$ in this embedding  is generated by quadrics. For $p >1$ the condition is that the first $p$ modules of syzygies among these quadrics are generated in minimal possible degree. The result  of Hwang and To in \cite{HT} is essentially  the case $p = 0$ of Theorem \ref{Main.Thm}.

In general it is very difficult to control Seshadri constants. However it was established in \cite{Lengths} that on an abelian variety they are related to a metric invariant introduced by Buser and Sarnak \cite{BS}. Specifically, write $A = V/\Lambda$, where $V$ is a complex vector space of dimension $g$ and $\Lambda \subseteq V$ is a lattice. Then $L$ determines a hermitian form $h = h_L$ on $V$, and  the Buser-Sarnak invariant is (the square of) the minimal length with respect to $h$ of a non-zero period of $\Lambda$:
\[   m(A, L) \ =_{\text{def}} \ \underset{0 \ne \ell \in \Lambda} \min  \, h_L(\ell, \ell). \]
The main result of \cite{Lengths} is that 
\[ \eps(A,L) \ \ge \ \frac{\pi}{4} \cdot m(A,L). \]
On the other hand, one can estimate $m(A,L)$ for very general $(A,L)$. In fact, suppose that  the polarization $L$ has elementary divisors \[  d_1 \, | \, d_2 \, | \, \ldots \, | \, d_g,  \] and put $d = d(L) = d_1 \cdot \ldots \cdot d_g$. By adapting an argument of Buser-Sarnak in the case of principal polarizations,
Bauer \cite{bauer} shows that if $(A,L)$ is very general, then
\[
m(A, L) \ \ge \ \frac{2^{1/g}}{\pi} \sqrt[^g]  {d \cdot g!}. 
\]
Therefore we obtain
\begin{corollaryalpha}
Assume that 
\[  d(L) \ > \ \frac{4^g (p+2)^g g^g}{2 g!} . \]
Then $(N_p)$ holds for very general $(A,L)$ of the given type.
\end{corollaryalpha}

The essential interest in statements of this sort occurs when $L$ is primitive (i.e. $d_1 = 1$), or at least when $d_1$ is small: as far as we know,  the present result is the first to give statements for higher syzygies of primitive line bundles in large dimension. By contrast, if $L$ is a suitable multiple of some ample line bundle, then much stronger statements are known. Most notably, the second author proved in \cite{Par} that $(N_p)$ always holds as soon as $d_1 \ge p +3$. This was strengthened and systematized in \cite{PP1} and \cite{PP2}, while (for $p = 0$) other statements appear in \cite{Iyer} and \cite{FG}.  

We conclude this Introduction by sketching a proof of the theorem of Hwang and To via the approach of the present paper. Following a time-honored device, one considers the diagonal $\Delta \subseteq A \times A$, with ideal sheaf $\II_\Delta$. Writing $L \boxtimes L  = \pro_1^* L \otimes \pro_2^*L$ for the exterior product of $L$ with itself, the essential point is to prove  \begin{equation} \HH{1}{A \times A}{L \boxtimes L \otimes \II_{\Delta}} \ = \ 0.
\tag{*} \end{equation}
Hwang and To achieve this by establishing a somewhat delicate upper bound on the volume of a one-dimensional analytic subvariety of a tubular neighborhood of  $\Delta$ (or more generally of a tubular neighborhood of any subtorus of an abelian variety). This allows them to control the positivity required to apply vanishing theorems on the blowup of $A \times A$ along $\Delta$. While their calculation is of substantial independent interest, for the task at hand it is considerably quicker to deduce (*) directly from Nadel vanishing.  

Specifically, using the hypothesis that $\eps(A,L) > 2g$, a standard argument (Lemma \ref{construct.single.divisor}) shows that for suitable $0 < c \ll 1$ one can construct an effective $\QQ$-divisor
\[  E_0 \ \num \ \left(\frac{ 1-c}{2}\right) L \]
on $A$ whose multiplier ideal vanishes precisely at the origin: $\MI{A, E_0} = \II_{0}$.
Now consider the difference map
\[  \delta : A \times A \lra A \ \ , \ \  (x,y) \mapsto x - y,   \]
and set $E = \delta^* E_0$. Since forming multiplier ideals commutes with pullback under smooth morphisms, we have on the one hand
\[   \MI{A \times A,E} \ = \ \delta^* \MI{A, E_0}  \ = \ \II_{\Delta}. \]
On the other hand,   one knows that 
\begin{equation}
L^2 \boxtimes L^2 \ = \ \delta^* (L) \otimes N 
\tag{**} \end{equation}
for a suitable nef line bundle $N$ on $A \times A$. Thanks to our choice of $E_0$, this implies   that $(L \boxtimes L) (-E)$ is ample. Therefore Nadel vanishing gives (*), as required. 

The proof of the general case of Theorem \ref{Main.Thm} proceeds along similar lines. Following an idea going back to Green \cite{Green}, one works on the $(p+2)$-fold product of $A$, where one has to check a vanishing involving the ideal sheaf of a union of pairwise diagonals.\footnote{The possibility of applying vanishing theorems on a blow-up to verify Green's criterion was noted already in \cite[Remark on p.\,600]{BEL}. Nowadays one can invoke the theory of   \cite{LiLi} to control the blow-ups involved: the pair-wise diagonals $\Delta_{0,1}, \ldots,  \Delta_{0,p+1}$ form a building set in the sense of \cite{LiLi} on the $(p+2)$-fold self product of a smooth variety. However in the case of abelian varieties treated here, elementary properties of multiplier ideals are used to obviate the need for any blowings up.} To realize this as a  multiplier ideal, we pull back a suitable divisor under a multi-subtraction map: this is carried out in \S 1. The positivity necessary for Nadel vanishing is verified using an analogue of (**)  established in \S2. Finally, \S3 contains some complements and variants, including a criterion for $L$ to define an embedding in which the homogeneous coordinate ring of $A$ is Koszul.

For applications of Nadel vanishing, one typically has to estimate the positivity of formal twists of line bundles by $\QQ$-divisors. To this end, we will allow ourselves to be a little sloppy in mixing additive and multiplicative notation. Thus given a $\QQ$-divisor $D$ and a line bundle $L$, the statement 
$ D \num b L $ is intended to mean that $D$ is numerically equivalent to $b \cdot c_1(L)$. Similarly, to say that $(b L )(-D)$ is ample indicates that $b \cdot c_1(L) - D$ is an ample numerical class. We trust that no confusion will result.

We are grateful to Thomas Bauer, Jun-Muk Hwang and Sam Payne for  valuable discussions.

\section{Proof of Theorem A}

As in the Introduction, let $A$ be an abelian variety of dimension $g$, and let $L$ be an ample line bundle on $A$. 

We start by recalling a geometric criterion that guarantees property $(N_p)$ in our setting. 
Specifically, form  the $(p+2)$-fold product  $X = A^{\times(p+2)}$ of $A$ with itself, and inside $X$ consider the reduced algebraic set  
 \begin{align*}
\Sigma  \   & =  \  \big\{ (x_0, \ldots, x_{p+1}) \mid x_0 = x_i \text{ for some $1 \le i \le p+1$ }   \big\}  \\  \ & =  \  \Delta_{0,1} \, \cup \, \Delta_{0,2} \, \cup  \, \ldots \, \cup \, \Delta_{0, p+1} \end{align*}
arising as the union of the  indicated pairwise diagonals. Thus $\Sigma$ has $p+1$ irreducible components, each of codimension $g$ in $X$. 

It was observed by Green \cite[\S 3]{Green} that property $(N_p)$ for $L$ is implied by a vanishing on $X$ involving the ideal sheaf of $\II_{\Sigma}$, generalizing the condition (*) in the Introduction for projective normality.   We refer to \cite{Ina} for a statement and  careful discussion of the criterion in general.\footnote{The argument appearing in \cite{Green} is somewhat oversimplified.} In the present situation, it shows that Theorem  \ref{Main.Thm} is a consequence of the following
\begin{proposition} \label{Basic.Vanishing}
Assume that $\eps(A,L) > (p+2)g$. Then 
\[
\HHH{i}{A^{\times (p+2)} }{ \opbox^{p+2}L   \otimes Q \otimes \II_{\Sigma} } \ = \ 0 
\]
for any nef line bundle $Q$ on $X$ and all $i > 0$.\footnote{As explained in \cite{Ina} one actually needs the vanishings:
\[ \HH{i}{A^{\times(p^\pr +2)}}{L^q \boxtimes L \boxtimes \ldots \boxtimes L \otimes \II_{\Sigma}}  \ = \ 0 \]
for $0 \le p^\pr \le p$ and $q \ge 1$, but these are all implied by the assertion of the Proposition.} \end{proposition}
 
 The plan is to deduce the proposition from Nadel vanishing. To this end, it suffices to produce an effective $\QQ$-divisor $E$ on $X$ having two properties:
\begin{gather}
\MI{X,E} \ = \ \II_{\Sigma} \label{first.eqn}
\\
\Big(\opbox^{p+2}L \Big)(-E) \ \text{ is ample}. \label{second.eqn}
\end{gather}
The rest of this section will be devoted to the construction of $E$ and the verification of these requirements. 
 
 The first point is quite standard:
  \begin{lemma} \label{construct.single.divisor}
 Assuming that $\eps(A,L) > (p+2)g$, there exists an effective $\QQ$-divisor $F_0$ on $A$ having the properties that 
 \[  F_0 \ \num \ \left (\frac{1-c}{p+2} \right) L \]
 for some $0 < c \ll 1$, and 
 \[  \MI{A,F_0} \ = \ \II_{0}. \]
 \end{lemma}
 \noi Here naturally  $\II_{0}\subseteq \OO_A$ denotes the ideal sheaf of the origin $0 \in A$. 
 \begin{proof}[Proof of Lemma]
We claim that for suitable $0 < c \ll 1$ and sufficiently divisible $k \gg 0$, there exists a divisor $D \in \linser{ k(1-c)L}$ with 
 \[
 \mult_0(D) \ = \ (p+2)g k ,
 \]
where in addition $D$ has a smooth tangent cone at the origin $0 \in A$ and is  non-singular away from $0$. Granting this, it suffices to put $F_0 = \tfrac{1}{(p+2)k}D$. As for the existence of $D$, let 
\[ \rho: A^\pr = \Bl_0(A) \lra A\] be the blowing up of $A$ at $0$, with exceptional divisor $T \subseteq A^\pr$. Then by definition of $\eps(A,L)$ the class $(1-c) \rho^*L  - (p+2)gT$ is ample on $A^\pr$ for $0 < c \ll 1$. If $D^\pr$ is a general divisor in the linear series  corresponding to a large multiple of this class, then Bertini's theorem on $A^\pr$ implies that $D = \rho_*(D^\pr) $ has the required properties.  \end{proof} 
  
  Now form the $(p+1)$-fold product $Y = A^{\times (p+1)}$ of $A$ with itself, and write $\pro_i : Y \lra A$ for the $i^{\text{th}}$ projection. Consider the reduced algebraic subset 
  \begin{align*}
 \Lambda \ &= \ \bigcup_{i = 1}^{p+1} \, \pro_i^{-1} ( \,  0 \,  ) \\
 & = \ \big \{ (y_1, \ldots, y_{p+1} ) \mid  y_i = 0 \text{ for some $1 \le i \le p+1 $
} \big \}. 
 \end{align*}
We wish to realize $\II_\Lambda$ as a multiplier ideal, to which end we simply consider the ``exterior sum" of the divisors $F_0$ just constructed.  Specifically, put
\[  E_0 \ = \ \sum_{i=1}^{p+1} \,  \pro_i^* (F_0) . \] 
Thanks to \cite[9.5.22]{PAG} one has
\[\MI{Y, E_0} \  = \ \prod_{i=1}^{p+1} \, \pro_i^* \MI{A, F_0} \ = \ \prod_{i=1}^{p+1}\, \pro_i^*\,  \II_{0},  
 \]
 i.e. $\MI{Y, E_0} = \II_{\Lambda}$, as desired. 
 
Next, consider the map
 \begin{equation}
 \delta = \delta_{p+1} : A^{\times ({p+2})} \lra A^{\times (p+1)} \ \ , \ \ (x_0, x_1, \ldots, x_{p+1}) \mapsto (x_0 - x_1, \ldots, x_0 - x_{p+1}), 
 \end{equation}
 and note that $\Sigma = \delta^{-1}\Lambda$ (scheme-theoretically). Set
 \[
E \ = \ \delta^* (E_0). 
 \]
 Since forming multiplier ideals commutes with pulling back under smooth morphisms (\cite[9.5.45]{PAG}), we find that
 \[
 \MI{X, E} \ = \ \delta^* \, \MI{Y , E_0} \ = \ \delta^* \II_{\Lambda} \ = \ \II_{\Sigma},
 \]
 and thus \eqref{first.eqn}
is satisfied.
 
 In order to verify \eqref{second.eqn}, we use the following assertion, which will be established in the next section.
 \begin{proposition} \label{pullback1} 
 There is a nef line bundle $N$ on $X = A^{\times (p+2)}$ with the property that
\begin{equation}\label{num.equiv.eqn} 
 \delta^* \Big( \opbox^{p+1} L \Big ) \,  \otimes N \ = \ \opbox^{p+2} \, L^{p+2} \, . 
\end{equation}  
 \end{proposition}
 Granting this, the property \eqref{second.eqn} -- and with it, Proposition \ref{Basic.Vanishing}
-- follows easily. Indeed, note that 
 \[ E\  \num \ \Big(\frac{1-c}{p+2}\Big) \cdot \Big( \delta^* \big(  \opbox^{p+1} L \big) \Big). \] 
Therefore \eqref{num.equiv.eqn} implies that
 \[
\Big(\opbox^{p+2}L \Big)(-E) \ \num \ c \cdot \Big(\opbox^{p+2}L \Big) \, + \, \Big(\frac{1-c}{p+2}\Big) \cdot N,
 \]
  which is ample. This completes the proof of Theorem \ref{Main.Thm}.

\section{Proof of Proposition \ref{pullback1}}

Let $A$ an abelian variety and $p$ a non-negative integer. Define the following maps:
$$b: A^{\times ({p+2})} \rightarrow A, \qquad  (x_0,x_1, \ldots, x_{p+1})\mapsto x_0 + x_1 + \ldots + x_{p+1}.$$
and for any $0 \le i < j \le p+1$ 
 $$d_{ij}: A^{\times ({p+2})}\rightarrow A, \qquad  (x_0,x_1, \ldots, x_{p+1})\mapsto x_i - x_j.$$
Recall the map $\delta$ from the previous section:
$$\delta: A^{\times ({p+2})} \lra A^{\times (p+1)} \ \ , \ \ (x_0, x_1, \ldots, x_{p+1}) 
\mapsto (x_0 - x_1, \ldots, x_0 - x_{p+1}).$$ 
Proposition \ref{pullback1} above follows from the following more precise statement.\footnote{Note that $L$ 
and $(-1)^*L$ differ by a topologically trivial line bundle.}

\begin{proposition}\label{pullback2}
For any ample line bundle $L$ on $A$ we have
$$ \delta^* \Big( \opbox^{p+1} L \Big ) \otimes \big( b^* L \big) \, \otimes \, \Big( \underset{1\le i < j}{\otimes} d_{ij}^*L \Big)
\ = \ \opbox^{p+1}_{k=0} \, \Big( L^{p+2-k}\otimes (-1)^* L^k \Big) \, .$$ 
\end{proposition}
Let 
$$a: A\times A \longrightarrow A ~~{\rm ~~and ~~} ~~ d: A\times A \longrightarrow A$$
be the addition and subtraction map respectively, $\cP$ a normalized Poincar\'e line bundle on $A\times \widehat{A}$, and 
$\phi_L: A \rightarrow \widehat{A}$ the isogeny induced by $L$. We use the notation
$$P=(1\times \phi_L)^*\cP \qquad {\rm and} \qquad P_{ij} = \pro_{ij}^* P,$$
where $\pro_{ij} : A^{\times ({p+2})} \rightarrow A\times A$ is the projection on the $(i,j)$-factor. We will use repeatedly 
the following standard facts.

\begin{lemma}\label{identities}
The following identities hold:

\noindent
$($i$).$ $a^*L \ \cong \ (L\boxtimes L) \otimes P$. 

\noindent
$($ii$)$. $d^*L \ \cong \ (L\boxtimes (-1)^*L) \otimes P^{-1}$.

\noindent
$($iii$)$. $\pro_{13}^* P \otimes \pro_{23}^* P 
\ \cong \ (a \times 1)^* P$  on the triple product $A\times A\times A$.
\end{lemma}
\begin{proof}
Identity (i) is well known (see e.g. \cite{mumford} p.78) and follows from the seesaw principle.  Identity (ii) can then be deduced from (i), by noting that $d = a \circ (1, -1)$. This gives
\begin{align*} d^*L \ &\cong \ (1\times (-1))^* \bigl( (L\boxtimes L) \otimes (1\times \phi_L)^* \cP\bigr) \\ &\cong \
(L\boxtimes (-1)^*L) \otimes (1\times ((-1) \circ \phi_L))^* \cP\\  &\cong  \
(L\boxtimes (-1)^*L) \otimes (1\times  \phi_{(-1)^*L})^* (1,-1)^* \cP\\ &\cong \ (L\boxtimes (-1)^*L) \otimes (1\times \phi_L)^*\cP^{-1},\end{align*}
where the last isomorphism follows from the well-known identity
$$((-1) \times 1)^* \cP\cong (1\times (-1))^* \cP \cong \cP^{-1}.$$
Identity (iii) follows from the formula
$$\pro_{13}^* \cP \otimes \pro_{23}^* \cP \cong (a,1)^* \cP$$
on $A\times A \times \widehat{A}$, which in turn is easily verified using the seesaw principle
(see e.g. the proof of Mukai's inversion theorem \cite[Theorem 2.2]{mukai}).
\end{proof}

Proposition \ref{pullback2} follows by putting together the formulas in the next Lemma.

\begin{lemma}\label{pullback3}
If $L$ is an ample line bundle on $A$, the following identities hold:

\noindent
$($i$)$. $b^* L \  \cong  \ \Big( ~\overset{p+2}{\opbox} L \Big )\otimes \Big( \underset{i<j}{\otimes} P_{ij} \Big)$.

\noindent
$($ii$)$.  $d_{ij}^*L \ \cong \ \Big( \OO_A \boxtimes \ldots\boxtimes \underset{i}{L} \boxtimes \ldots \boxtimes \underset{j}{(-1)^* L}\boxtimes \ldots \boxtimes \OO_A \Big ) \otimes P_{ij}^{-1}$, for all $i < j$.

\noindent
$($iii$)$. $\delta^* \Big( ~\overset{p+1}{\opbox} L \Big )  \ \cong  \ \Big( L^{p+1} \boxtimes (-1)^*L\boxtimes \ldots \boxtimes (-1)^*L \Big) 
\otimes P_{01}^{-1} \otimes \ldots \otimes P_{0,p+1}^{-1}$.
\end{lemma}
\begin{proof}
(i) If $p = 0$ this is Lemma \ref{identities} (i). We can inductively obtain the formula for some $p >0$ from that for $p-1$ 
by noting that $b ~(= b_{p+2}) = (a, {\rm id}) \circ b_{p+1}$, where $b_k$ denotes the addition map for $k$ factors, $a$ is addition 
map on the first two factors, and ${\rm id}$ is the identity on the last $p$ factors.  Therefore inductively we have 
$$b^*L \cong (a, {\rm id})^* \Big( \big( ~\overset{p+1}{\opbox} L\big) \otimes \big( \underset{i<j}{\otimes} P_{ij} \big)\Big).$$
The formula follows then by using Lemma \ref{identities} (i) for the addition map $a$ on the first two factors, and Lemma \ref{identities}
(iii) for the combination of the first two factors with any of the other $p$ factors.

\noindent
(ii) This follows simply by noting that $d_{ij} = d \circ p_{ij}$, where $p_{ij}$ is the projection on the $(i,j)$ factors and $d$ is the difference 
map. We then apply Lemma \ref{identities} (ii).

\noindent
(iii) Note that $\delta = (d_{01}, \ldots, d_{0, p+1})$. Therefore
$$\delta^* \Big( ~\overset{p+1}{\opbox} L \Big )  \cong  d_{01}^* L \otimes \ldots \otimes d_{0,p+1}^* L.$$
One then applies the formula in (ii).
\end{proof}

In order to discuss the Koszul property in the next section, we will need a  variant of these results. Specifically, fix $k \ge 2$ and consider the mapping
$$\gamma: A^{\times k} \lra A^{\times (k-1)} \ \ , \ \ (x_0, x_1, \ldots, x_{k}) 
\mapsto (x_0 - x_1, x_1 - x_2, \ldots, x_{k-1} - x_{k}).$$ 
Consider also for any $0 \le i < j \le k$ the maps
 $$a_{ij}: A^{\times k }\rightarrow A, \qquad  (x_0,x_1, \ldots, x_{k})\mapsto x_i + x_j.$$

\begin{variant}\label{pullback4}
For any ample line bundle $L$ on $A$ we have
\small
$$ \gamma^* \Big( \opbox^{k} L \Big ) \, \otimes \, \Big( \underset{0\le i \le  k-1}{\otimes} a_{i,i+1}^*L \Big) 
\ = \  L^2 \opbox \Big(L^2\otimes (-1)^*L \Big) \opbox\ldots \opbox 
\Big(L^2\otimes (-1)^*L \Big) \opbox \Big(L \otimes (-1)^*L \Big) \, .$$ 
\end{variant}
\normalsize
\begin{proof}
Noting that $a_{ij} = a \circ \pro_{ij}$, where $\pro_{ij}$ is the projection on the $(i,j)$ factors and $a$ is the difference 
map, and using Lemma \ref{identities} (i), we have 
$$a_{i j}^* L \cong \ \Big( \OO_A \boxtimes \ldots\boxtimes \underset{i}{L} \boxtimes \ldots \boxtimes \underset{j}{L}\boxtimes \ldots \boxtimes \OO_A \Big ) \otimes P_{ij} \, .$$ 
On the other hand,  $\gamma = (d_{01}, d_{12}, \ldots, d_{k-1,k})$ and using Lemma \ref{pullback3} (ii) for each of the 
factors, we have
$$\gamma^* \Big( \opbox^{k-1} L \Big ) \cong \Big( L\boxtimes \big(L\otimes (-1)^*L\big) \boxtimes  \ldots \boxtimes 
\big(L\otimes (-1)^*L \big) \boxtimes (-1)^*L \Big) \otimes P_{01}^{-1} \otimes \ldots \otimes P_{k-1,k}^{-1} \, .$$
\end{proof}

\begin{corollary} \label{Kosz.Pos.Corollary}
 There is a nef line bundle $N$ on $A^{\times k}$ with the property that
\begin{equation}
 \gamma^* \Big( \opbox^{k-1} L \Big ) \,  \otimes N \ = \ \opbox^{k} \, L^{3} \, \qed
\end{equation}  
\end{corollary}

\section{Complements}

This section contains a couple of additional results that are established along the same lines as those above. As before $A$ is an abelian variety of dimension $g$, and $L$ is an ample line bundle on $A$.

We start with a criterion for $L$ to define an embedding in which $A$ satisfies the Koszul property.

\begin{proposition}
Assume that $\eps(A,L) > 3g$. Then under the embedding defined by $L$, the homogeneous coordinate ring of $A$ is a Koszul algebra.
\end{proposition}

\begin{proof} [Sketch of Proof]
Fix $k \ge 2$, and consider the $k$-fold self product $A^{\times k}$ of $A$.  By analogy to Green's criterion, it is known that the Koszul property is implied by the vanishings (for all $k \ge 2$) 
\begin{equation}
\HHH{1}{A^{\times k}}{\opbox^k L \otimes Q \otimes \II_{\Gamma}} \ = \ 0
\end{equation}
where $Q$ is a nef bundle on $A^{\times k}$, and $\Gamma$ is the reduced algebraic set:
\[  \Gamma   =    \Delta_{1,2}    \cup \, \Delta_{2,3} \, \cup \, \ldots \, \cup \, \Delta_{k-1,k}  \]
(see \cite[Proposition 1.9]{IM}). As above, this is established by realizing $\Gamma$ as a  multiplier ideal and applying Nadel vanishing. For the first point, one constructs (as in the case $p = 2$ of Theorem \ref{Main.Thm})  a divisor $F_0 \num (\tfrac{1-c}{3}) L$ on $A$, takes its exterior sum on $A^{\times (k-1)}$, and then  pulls back under the map $\gamma: A^{\times k} \lra A^{\times (k-1)}$ appearing at the end of the last section. The required positivity follows from Corollary \ref{Kosz.Pos.Corollary}. 
\end{proof}

Finally, we record an analogue of the result of Hwang and To for Wahl maps (see \cite{wahl}).
\begin{proposition}
Let $L$ be an ample line bundle on $A$, and assume that $\eps(A, L) > 2(g+m)$ for some integer $m \ge 0$. Then
\[
\HHH{1}{A \times A}{ L \boxtimes L \otimes \II_{\Delta}^{m+1}} \ = \ 0.
\] 
In particular, the m-th Wahl (or Gaussian) map
\[  \gamma^m_L:\HHH{0}{A\times A} {L\boxtimes L\otimes\II_{\Delta}^m}\lra \HHH{0}{A\times A} {L\boxtimes L\otimes\II_{\Delta}^m\otimes\OO_\Delta}\cong \HHH{0}{A} {L^2\otimes S^m\Omega^1_A}\]
is surjective.
 \end{proposition}

\begin{proof}[Sketch of Proof]  One proceeds just as in the proof outlined in the Introduction, except that the stronger numerical hypothesis on $\eps(A,L)$ allows one to take 
$ E_0   \num   \left(\tfrac{ 1-c}{2}\right) L $
with $\MI{A, E_0} \ = \ \II_{0}^{m+1}$. For the rest one argues as before.
\end{proof}

\begin{remark} 
The previous proposition, combined with Bauer's result mentioned in the Introduction, and with Theorem B of \cite{cfp}, implies the surjectivity of the first Wahl map of curves of genus $g$ sitting on very general abelian surfaces for all $g>145$. This provides a ``non-degenerational" proof -- in the range $g>145$ -- of the surjectivity of the map $\gamma^1_{K_C}$   for general curves of genus $g$, which holds   for all $g\ge 12$ and $g=10$  (\cite{chm}) .
\end{remark}

\end{document}